\documentclass[a4paper,12pt,reqno]{amsart}
\usepackage[english]{babel}
\usepackage{amssymb,amsthm,amsmath,amsfonts}
\usepackage[numbers]{natbib}
\usepackage[colorlinks,citecolor=blue,urlcolor=blue]{hyperref}
\usepackage{hypernat}
\usepackage[top=2cm, bottom=2cm, left=2cm, right=2cm]{geometry}
\usepackage{graphicx}
\def\Prob{{\sf P}}

\begin{document}

\newtheorem{theorem}{Theorem}
\newtheorem{proposition}{Proposition}
\newtheorem{corollary}{Corollary}
\theoremstyle{remark}
\newtheorem*{rem}{Remark}

\dedicatory{Dedicated to Boris Mikhailovich Makarov, with great respect}

\title[Comparison theorems]{COMPARISON THEOREMS \\ FOR THE SMALL BALL PROBABILITIES \\ OF GAUSSIAN PROCESSES IN WEIGHTED $L_2$-NORMS}

\author{Alexander I. Nazarov}

\address{St. Petersburg Departament of Steklov Institute RAS, Fontanka 27, 191023, St. Petersburg, RUSSIA;
St.Petersburg State University, Universitetskii pr. 28, 198504, St.Petersburg, RUSSIA}

\email{al.il.nazarov@gmail.com}

\author{Ruslan S. Pusev}

\address{St.Petersburg State University, Universitetskii pr. 28, 198504, St.Petersburg, RUSSIA}

\email{Ruslan.Pusev@math.spbu.ru}

\subjclass[2000]{60G15}

\keywords{Small ball probabilities, Gaussian processes, comparison theorems, spectral asymptotics}

\begin{abstract}
We prove comparison theorems for small ball probabilities of the Green Gaussian processes in weighted $L_2$-norms. We find the sharp small ball asymptotics
for many classical processes under quite general assumptions on the weight.
\end{abstract}

\thanks{Authors are supported by RFBR grant 10-01-00154. The first author is also supported by St. Petersburg State University grant 6.38.64.2012. 
The second author is also supported by the Chebyshev Laboratory of St. Petersburg State University with the Russian Government grant 11.G34.31.0026, 
and by the Program of supporting for Leading Scientific Schools (NSh-1216.2012.1).}


\date{}

\maketitle

\section{Introduction}

The problem of small ball behavior of a random process $X$ in the norm $\|\cdot\|$ is to describe the asymptotics as $\varepsilon\to0$ of the probability 
$\Prob\{\|X\|\leq\varepsilon\}$. The theory of small ball behavior for Gaussian processes in various norms
is intensively developed in recent decades, see surveys \cite{Lifs:1999}, \cite{Li:Shao:2001} and the site \cite{Lifs:2010}.

Suppose we have a Gaussian process $X(t)$, $0 \leqslant t \leqslant 1$, with zero mean and covariance 
function $G_X(t,s)={\sf E}X(t)X(s)$, $t,s\in [0,1]$. Let $\psi$ be a non-negative weight function on $[0,1]$. We set
$$
\|X\|_\psi=\left(
\int_0^1 X^2(t)\psi(t)dt
\right)^{\frac12}
$$
(we drop the subscript $\psi$ if $\psi\equiv1$).

By the classical Karhunen--Lo\`eve expansion, one has the equality in distribution
$$
\|X\|_\psi^2  = \sum_{j=1}^\infty\lambda_j\xi_j^2.
$$
Here  $\xi_j$, $j\in\mathbb N$, are independent standard 
Gaussian random variables while $\lambda_j>0$, $j\in\mathbb N$, 
are the eigenvalues of the integral equation
$$
\lambda f(t)=\int_0^1 G(t,s)\sqrt{\psi(t)\psi(s)}f(s)\,ds,\quad t\in[0;1].
$$

In the papers \cite{Naza:2009,Naza:Niki:2004} there was selected the concept of the {\it Green} process, i.e. Gaussian process with covariance being 
the Green function for a self-adjoint differential operator. The approach developed in these papers allows to obtain the sharp (up to a constant) 
asymptotics of small deviations in $L_2$-norm for this class of processes. In the papers \cite{Naza:2003,Naza:Puse:2009}, using this approach, we
have calculated the sharp asymptotics of small ball probabilities for a large class of particular processes with various weights.

In this paper we prove a comparison theorem for the small ball probabilities of the Green Gaussian processes in the weighted $L_2$-norms. 
This theorem gives us the opportunity to obtain the sharp small ball asymptotics for many classical processes under quite general assumptions on the weight.
For the Wiener process and some other processes this result was obtained in \cite{Niki:Puse:2012}.

Let us recall some notation. A function $G(t,s)$ is called the Green function 
of a boundary value problem for differential operator $L$ if it satisfies the equation 
$LG=\delta(t-s)$ in the sense of distributions and satisfies the boundary conditions. 

The space $W_p^m(0,1)$ is the Banach space of functions $y$ having
continuous derivatives up to $(m-1)$-th order when $y^{(m-1)}$ is
absolutely continuous on $[0,1]$ and $y^{(m)}\in L_p(0,1)$. 

$\mathcal V(\dots)$ stands for the Vandermonde determinant.

\section{The calculation of the perturbation determinant}

Let $L$ be a self-adjoint differential operator of order $2n$, generated by the differential expression
\begin{equation}
\label{diffExpression}
Lv\equiv
(-1)^n v^{(2n)}+\left(p_{n-1}v^{(n-1)}\right)^{(n-1)}+\ldots+p_0v;
\end{equation}
and boundary conditions
\begin{equation}
\label{boundaryConditions}
U_\nu(v)\equiv U_{\nu 0}(v)+U_{\nu 1}(v)=0,
\quad
\nu=1,\ldots,2n.
\end{equation}
Here
$$
\begin{aligned}
U_{\nu 0}(v) = \alpha_\nu v^{(k_\nu)}(0)+\sum_{j=0}^{k_\nu-1}\alpha_{\nu j}v^{(j)}(0),\\
U_{\nu 1}(v) = \gamma_\nu v^{(k_\nu)}(1)+\sum_{j=0}^{k_\nu-1}\gamma_{\nu j}v^{(j)}(1),
\end{aligned}
$$
and for any $\nu$ at least one of coefficients $\alpha_\nu$ and $\gamma_\nu$ is not zero.

We assume that the system of boundary conditions (\ref{boundaryConditions}) is normalized. This means that the sum of orders of all boundary conditions 
$\varkappa=\sum_{\nu} k_{\nu}$ is minimal. See \cite[\S4]{Naim:1969}; see also \cite{Shka:1982} where a more general class of boundary value problems is considered.

We introduce the notation
$$
\widetilde\alpha_\nu=\alpha_\nu(\psi(0))^{\frac{k_\nu}{2n}-\frac{2n-1}{4n}},
\quad
\widetilde\gamma_\nu=\gamma_\nu(\psi(1))^{\frac{k_\nu}{2n}-\frac{2n-1}{4n}},
\quad
\omega_k=\exp(ik\pi/n),
$$

$$
\theta_1(\psi)=
\det
\mbox{\tiny$
\begin{pmatrix}
\widetilde\gamma_1 & \widetilde\alpha_1\omega_1^{k_1} & \ldots & \widetilde\alpha_1\omega_{n-1}^{k_1} & \widetilde\alpha_1\omega_n^{k_1} & \widetilde\gamma_1\omega_{n+1}^{k_1} & \ldots & \widetilde\gamma_1\omega_{2n-1}^{k_1}\\
\vdots & \vdots & \ddots & \vdots & \vdots & \vdots & \ddots & \vdots\\
\widetilde\gamma_{2n} & \widetilde\alpha_{2n}\omega_1^{k_{2n}} & \ldots & \widetilde\alpha_{2n}\omega_{n-1}^{k_{2n}} & \widetilde\alpha_{2n}\omega_n^{k_{2n}} & \widetilde\gamma_{2n}\omega_{n+1}^{k_{2n}} & \ldots & \widetilde\gamma_{2n}\omega_{2n-1}^{k_{2n}}\\
\end{pmatrix}
$},
$$

$$
\theta_{-1}(\psi)=
\det
\mbox{\tiny$
\begin{pmatrix}
\widetilde\alpha_1 & \widetilde\alpha_1\omega_1^{k_1} & \ldots & \widetilde\alpha_1\omega_{n-1}^{k_1} & \widetilde\gamma_1\omega_n^{k_1} & \widetilde\gamma_1\omega_{n+1}^{k_1} & \ldots & \widetilde\gamma_1\omega_{2n-1}^{k_1}\\
\vdots & \vdots & \ddots & \vdots & \vdots & \vdots & \ddots & \vdots\\
\widetilde\alpha_{2n} & \widetilde\alpha_{2n}\omega_1^{k_{2n}} & \ldots & \widetilde\alpha_{2n}\omega_{n-1}^{k_{2n}} & \widetilde\gamma_{2n}\omega_n^{k_{2n}} & \widetilde\gamma_{2n}\omega_{n+1}^{k_{2n}} & \ldots & \widetilde\gamma_{2n}\omega_{2n-1}^{k_{2n}}\\
\end{pmatrix}
$}.
$$

\begin{theorem}
\label{spectral}
Let $L$ be a self-adjoint differential operator of order $2n$, generated by the differential expression~(\ref{diffExpression})
and boundary conditions~(\ref{boundaryConditions}). Let also $p_m\in W_\infty^m(0,1)$, $m=0,\ldots,n-1$.

Consider two eigenvalue problems
\begin{equation}
\label{BVP}
Ly=\mu\psi_{1,2}y;\qquad
U_\nu(y)=0,\quad \nu=1,\ldots,2n,
\end{equation}
where $\psi_1$, $\psi_2\in W_\infty^n(0,1)$. Suppose that the weight functions $\psi_1$, $\psi_2$ are bounded away from zero, and
\begin{equation}
\label{J1=J2}
\int_0^1\psi_1^{\frac1{2n}}(x)dx=\int_0^1\psi_2^{\frac1{2n}}(x)dx=\vartheta.
\end{equation}

Denote by $\mu_k^{(j)}$, $j=1,2$, $k\in\mathbb N$, the eigenvalues of the problems (\ref{BVP}), enumerated in ascending order according to the multiplicity.
Then
$$
\prod_{k=1}^\infty\frac{\mu_k^{(1)}}{\mu_k^{(2)}}=
\left|\frac{\theta_{-1}(\psi_2)}{\theta_{-1}(\psi_1)}\right|.
$$
\end{theorem}

\begin{proof}
Consider the first problem in (\ref{BVP}). Denote by $\varphi_j(t,\zeta)$, $j=0,\ldots,2n-1$, solutions of the equation $Ly=\zeta^{2n}\psi_{1}y$, 
specified by the initial conditions $\varphi_j^{(k)}(0,\zeta)=\delta_{jk}$.

We substitute a general solution of the equation $y(t)=c_0\varphi_0(t,\zeta)+\ldots+c_{2n-1}\varphi_{2n-1}(t,\zeta)$ to the boundary conditions and
obtain $\mu_k^{(1)}=x_k^{2n}$, where $x_1\leq x_2\leq\ldots$ are positive roots of the function
$$
F_1(\zeta)=\det
\begin{pmatrix}
U_1(\varphi_0) & \ldots & U_1(\varphi_{2n-1})\\
\vdots & \ddots & \vdots\\
U_{2n}(\varphi_0) & \ldots & U_{2n}(\varphi_{2n-1})\\
\end{pmatrix}.
$$
It is easy to see (\cite[\S 2]{Naim:1969}) that $F_1(\zeta)$ is an entire function.

It is well known (see~\cite{Fedo:1993}, \cite[\S 4]{Naim:1969}) that there exist solutions $\widetilde\varphi_j(t,\zeta)$, $j=0,\ldots,2n-1$, of the
equation $Ly=\zeta^{2n}\psi_{1}y$ such that for large $|\zeta|$, $|\arg(\zeta)|\leqslant\frac\pi{2n}$, the following asymptotic relation holds:
\begin{equation}
\label{tildephi}
\widetilde\varphi_j(t,\zeta)=(\psi_1(t))^{-\frac{2n-1}{4n}}\exp\left(i\omega_j\zeta\int_0^t\psi_1^{\frac1{2n}}(u)du\right)\left(1+O(|\zeta|^{-1})\right),
\quad
j=0,\ldots,2n-1.
\end{equation}
The relation (\ref{tildephi}) is uniform in $t\in[0,1]$, and one can differentiate it.

It is easy to see that for $|\arg(\zeta)|\leqslant\frac\pi{2n}$, $|\zeta|\to\infty$
$$
U_\nu(\widetilde\varphi_j)=
\left(\alpha_\nu\widetilde\varphi_j^{(k_\nu)}(0,\zeta)+\gamma_\nu\widetilde\varphi_j^{(k_\nu)}(1,\zeta)\right)
\left(1+O(|\zeta|^{-1})\right).
$$

For large $|\zeta|$, the functions $\widetilde\varphi_j(t,\zeta)$ are linearly independent. Therefore there exists a matrix 
$C(\zeta)=(c_{jk})_{0\leq j,k\leq 2n-1}$ depending on $\zeta$ such that
$$
(\varphi_0(t,\zeta),\ldots,\varphi_{2n-1}(t,\zeta))^\top=C(\zeta)(\widetilde\varphi_0(t,\zeta),\ldots,\widetilde\varphi_{2n-1}(t,\zeta))^\top.
$$
Thus,
\begin{equation}
\label{F1}
F_1(\zeta)=\det(C(\zeta))\cdot\det
\begin{pmatrix}
U_1(\widetilde\varphi_0) & \ldots & U_{2n}(\widetilde\varphi_0)\\
\vdots & \ddots & \vdots\\
U_1(\widetilde\varphi_{2n-1}) & \ldots & U_{2n}(\widetilde\varphi_{2n-1})\\
\end{pmatrix}.
\end{equation}

By the initial conditions we have
$$
I_{2n}=C(\zeta)
\begin{pmatrix}
\widetilde\varphi_0(0,\zeta) & \ldots & \widetilde\varphi_0^{(2n-1)}(0,\zeta)\\
\vdots & \ddots & \vdots\\
\widetilde\varphi_{2n-1}(0,\zeta) & \ldots & \widetilde\varphi_{2n-1}^{(2n-1)}(0,\zeta)\\
\end{pmatrix}.
$$

By the relations (\ref{tildephi}), we obtain for $|\arg(\zeta)|\leqslant\frac\pi{2n}$, $|\zeta|\to\infty$
\begin{multline*}
\det
\begin{pmatrix}
\widetilde\varphi_0(0,\zeta) & \ldots & \widetilde\varphi_0^{(2n-1)}(0,\zeta)\\
\vdots & \ddots & \vdots\\
\widetilde\varphi_{2n-1}(0,\zeta) & \ldots & \widetilde\varphi_{2n-1}^{(2n-1)}(0,\zeta)\\
\end{pmatrix}
=(\psi_1(0))^{2n\left(-\frac12+\frac1{4n}\right)}
\times\\\times
\det
\begin{pmatrix}
1 & \left(i\zeta(\psi_1(0))^{\frac1{2n}}\right)^1 & \ldots & \left(i\zeta(\psi_1(0))^{\frac1{2n}}\right)^{2n-1}\\
\vdots & \vdots & \ddots & \vdots\\
1 & \left(i\omega_{2n-1}\zeta(\psi_1(0))^{\frac1{2n}}\right)^1 & \ldots & \left(i\omega_{2n-1}\zeta(\psi_1(0))^{\frac1{2n}}\right)^{2n-1}\\
\end{pmatrix}
(1+O(|\zeta|^{-1}))
=\\=
(\psi_1(0))^{\frac{1-2n}2}
\left(i\zeta(\psi_1(0))^{\frac1{2n}}\right)^{n(2n-1)}
\mathcal V(1,\omega_1,\ldots,\omega_{2n-1})
(1+O(|\zeta|^{-1}))
=\\=
\left(i\zeta\right)^{2n^2-n}
\mathcal V(1,\omega_1,\ldots,\omega_{2n-1})
(1+O(|\zeta|^{-1})).
\end{multline*}
Whence, for $|\arg(\zeta)|\leqslant\frac\pi{2n}$, $|\zeta|\to\infty$, we have
$$
\det(C(\zeta)) = \frac{\left(i\zeta\right)^{n-2n^2}}{\mathcal V(1,\omega_1,\ldots,\omega_{2n-1})}\cdot(1+O(|\zeta|^{-1})).
$$

Next, following \cite[\S 4]{Naim:1969}, we obtain for $|\arg(\zeta)|\leqslant\frac\pi{2n}$, $|\zeta|\to\infty$
\begin{multline*}
\det
\begin{pmatrix}
U_1(\widetilde\varphi_0) & \ldots & U_1(\widetilde\varphi_{2n-1})\\
\vdots & \ddots & \vdots\\
U_{2n}(\widetilde\varphi_0) & \ldots & U_{2n}(\widetilde\varphi_{2n-1})\\
\end{pmatrix}
=\\=
(i\zeta)^{\varkappa}
\exp\left(-i\omega_1\vartheta\zeta-i\omega_2\vartheta\zeta-\ldots-i\omega_{n-1}\vartheta\zeta\right)
\times\\\times
\left(\theta_1(\psi_1)\exp(i\vartheta\zeta)+\theta_0(\psi_1)+\theta_{-1}(\psi_1)\exp(-i\vartheta\zeta)\right)
(1+O(|\zeta|^{-1}))
\end{multline*}
(we recall that $\varkappa=k_1+\ldots+k_{2n}$), where $\theta_0(\psi_1)$ is some unimportant constant.

It is easy to see (\cite[Theorem 1.1]{Naza:2009}) that $\theta_1(\psi_1)=-\omega_1^\varkappa\theta_{-1}(\psi_1)$.

Substituting these formulas to (\ref{F1}) we obtain for $|\arg(\zeta)|\leqslant\frac\pi{2n}$, $|\zeta|\to\infty$
\begin{multline*}
F_1(\zeta)=
\frac
{\left(i\zeta\right)^{n-2n^2+\varkappa}\exp\left(-i\omega_1\vartheta\zeta-i\omega_2\vartheta\zeta-\ldots-i\omega_{n-1}\vartheta\zeta\right)}
{\mathcal V(1,\omega_1,\ldots,\omega_{2n-1})}
\times\\\times
\left(\theta_{-1}(\psi_1)(\exp(-i\vartheta\zeta)-\omega_1^\varkappa\exp(i\vartheta\zeta))+\theta_0(\psi_1)\right)
(1+O(|\zeta|^{-1})).
\end{multline*}

Now we consider the second problem in (\ref{BVP}) and define the function $F_2(\zeta)$ similarly to $F_1(\zeta)$ with $\psi_2$ instead of $\psi_1$. 
Then the following relation holds:
\begin{multline*}
F_2(\zeta)=
\frac
{\left(i\zeta\right)^{n-2n^2+\varkappa}\exp\left(-i\omega_1\vartheta\zeta-i\omega_2\vartheta\zeta-\ldots-i\omega_{n-1}\vartheta\zeta\right)}
{\mathcal V(1,\omega_1,\ldots,\omega_{2n-1})}
\times\\\times
\left(\theta_{-1}(\psi_2)(\exp(-i\vartheta\zeta)-\omega_1^\varkappa\exp(i\vartheta\zeta))+\theta_0(\psi_2)\right)
(1+O(|\zeta|^{-1})).
\end{multline*}

Whence, for $|\zeta|\to\infty$, $\arg(\zeta)\neq\frac{\pi j}{n}$, $j\in\mathbb Z$, we obtain
$$
\left|\frac{F_2(\zeta)}{F_1(\zeta)}\right|\rightarrow
\left|\frac{\theta_{-1}(\psi_2)}{\theta_{-1}(\psi_1)}\right|.
$$
Moreover, the quotient $\left|{F_2(\zeta)}/{F_1(\zeta)}\right|$ is uniformly bounded on circles $|\zeta|=r_k$ for a proper sequence $r_k\to\infty$.

Further, by continuity of solutions to a differential equation  with respect to parameters, we have $F_1(\zeta)/F_2(\zeta)\rightrightarrows1$ as $\zeta\to0$.

Applying the Jensen Theorem to $F_1(\zeta)$ and $F_2(\zeta)$, we obtain
$$
\prod_{k=1}^\infty\frac{\mu_k^{(1)}}{\mu_k^{(2)}}=
\exp\left(\lim_{\rho\to\infty}\frac1{2\pi}\int_0^{2\pi}\ln\frac{|F_2(\rho e^{i\theta})|}{|F_1(\rho e^{i\theta})|}d\theta\right)=
\left|\frac{\theta_{-1}(\psi_2)}{\theta_{-1}(\psi_1)}\right|.
$$
\end{proof}

\begin{corollary}
\label{main}
Let the covariance of a centered Gaussian process~$X(t)$,
$0\leqslant t\leqslant1$, be the Green function of a self-adjoint operator~$L$ generated by the differential expression~(\ref{diffExpression})
and boundary conditions~(\ref{boundaryConditions}).

Let the coefficients $p_m$, $m=0,\ldots,n-1$, and the weight functions $\psi_1$, $\psi_2$ satisfy the assumptions of Theorem~\ref{spectral}. Then
$$
\lim_{\varepsilon\to0}
\frac{\Prob(\|X\|_{\psi_1}\leq\varepsilon)}{\Prob(\|X\|_{\psi_2}\leq\varepsilon)}=
\left|\frac{\theta_{-1}(\psi_2)}{\theta_{-1}(\psi_1)}\right|^{1/2}.
$$
\end{corollary}

\begin{proof}
Denote by $\mu_k^{(1,2)}$ the eigenvalues of the problems (\ref{BVP}).

Using the Li comparison theorem (see \cite{Li:1992,Gao:Hann:Torc:2003a}) and Theorem~\ref{spectral}, we obtain
$$
\lim_{\varepsilon\to0}\frac{\Prob(\|X\|_{\psi_1}\leq\varepsilon)}{\Prob(\|X\|_{\psi_2}\leq\varepsilon)}=
\left(\prod_{k=1}^\infty\frac{\mu_k^{(1)}}{\mu_k^{(2)}}\right)^{\frac12}=
\left|\frac{\theta_{-1}(\psi_2)}{\theta_{-1}(\psi_1)}\right|^{\frac12}.
$$
\end{proof}

\begin{rem}
If the assumption (\ref{J1=J2}) does not hold then the probabilities $\Prob(\|X\|_{\psi_{1,2}}\leq\varepsilon)$ have different logarithmic asymptotics (see 
\cite[Theorem 7.3]{Naza:Niki:2004}).
\end{rem}

\section{Separated boundary conditions}

Now we consider an important particular case.

\begin{theorem}
\label{sep}
Let the assumptions of Corollary~\ref{main} be satisfied. Suppose also that the boundary conditions~(\ref{boundaryConditions}) are separated in main terms, 
i.e. have the form
$$
\left.
\begin{aligned}
v^{(k_\nu)}(0)+\sum_{j=0}^{k_\nu-1}\left(\alpha_{\nu j}v^{(j)}(0)+\gamma_{\nu j}v^{(j)}(1)\right)=0,\\
v^{(k'_\nu)}(1)+\sum_{j=0}^{k'_\nu-1}\left(\alpha'_{\nu j}v^{(j)}(0)+\gamma'_{\nu j}v^{(j)}(1)\right)=0,
\end{aligned}
\right\}
\quad
\nu=1,\ldots,n.
$$
Denote by $\varkappa_0$ and $\varkappa_1$ sums of orders of boundary conditions at zero and one, respectively:
$\varkappa_0=k_1+\ldots+k_n$, $\varkappa_1=k'_1+\ldots+k'_n$.
Then
\begin{equation}
\label{separated}
\lim_{\varepsilon\to0}
\frac{\Prob(\|X\|_{\psi_1}\leq\varepsilon)}{\Prob(\|X\|_{\psi_2}\leq\varepsilon)}=
\left(
\frac{\psi_2(0)}{\psi_1(0)}
\right)
^{-\frac{n}4+\frac18+\frac{\varkappa_0}{4n}}
\left(
\frac{\psi_2(1)}{\psi_1(1)}
\right)
^{-\frac{n}4+\frac18+\frac{\varkappa_1}{4n}}.
\end{equation}
\end{theorem}

\begin{proof}
Under assumptions of the Theorem the matrix determining $\theta_{-1}(\psi)$ is block diagonal, and we obtain
$$
\theta_{-1}(\psi)=
(-1)^{\varkappa_1}
(\psi(0))^{{\frac{\varkappa_0}{2n}-\frac{2n-1}{4}}}
(\psi(1))^{{\frac{\varkappa_1}{2n}-\frac{2n-1}{4}}}
\cdot
\mathcal V(\omega_1^{k_1},\ldots,\omega_1^{k_n})
\cdot
\mathcal V(\omega_1^{k'_1},\ldots,\omega_1^{k'_n}).
$$

Therefore,
$$
\frac{\theta_{-1}(\psi_2)}{\theta_{-1}(\psi_1)}
=
\left(\frac{\psi_2(0)}{\psi_1(0)}\right)^{\frac{\varkappa_0}{2n}-\frac{2n-1}{4}}
\left(\frac{\psi_2(1)}{\psi_1(1)}\right)^{\frac{\varkappa_1}{2n}-\frac{2n-1}{4}}.
$$
\end{proof}

Many classical Gaussian processes satisfy the assumptions of Theorem \ref{sep}. We give several examples.

For a random process $X(t)$, $0\leq t\leq 1$, denote by $X_m^{[\beta_1,\ldots,\beta_m]}(t)$, $0\leq t\leq 1$, the $m$-times integrated process:
$$
X_m^{[\beta_1,\ldots,\beta_m]}(t)=(-1)^{\beta_1+\ldots+\beta_m}\int_{\beta_m}^t\dots\int_{\beta_1}^{t_1}X(s)dsdt_1\ldots
$$
(any index $\beta_\nu$ equals 0 or 1).

Following \cite{Naza:2003}, we introduce the notation
$$
z_n=\exp(i\pi/n),
\quad
\varepsilon_n=\left(\varepsilon\sqrt{2n\sin\frac\pi{2n}}\right)^{\frac1{2n-1}},
\quad
\mathcal D_n=\frac{2n-1}{2n\sin\frac\pi{2n}}.
$$

\begin{proposition}
\label{wiener}
Suppose that the function $\psi\in W_\infty^{m+1}(0,1)$ is bounded away from zero and satisfies the relation $\int_0^1\psi^{\frac1{2(m+1)}}(x)dx=1$. 
Then for the integrated Brownian motion the following relation holds:
\begin{multline*}
\Prob(\|W_m^{[\beta_1,\ldots,\beta_m]}\|_{\psi}\leq\varepsilon)
\sim
\left(
\frac{\psi(1)}{\psi(0)}
\right)
^{-\frac{m+1}8+\frac{\mathcal K}{4(m+1)}}
\times\\\times
\frac{(2m+2)^{\frac{m}2+1}}{|\mathcal V(1,z_{m+1}^{1-3\beta_1},z_{m+1}^{2-5\beta_2},\ldots,z_{m+1}^{m+\beta_m})|}
\frac{\varepsilon_{m+1}}{\sqrt{\pi\mathcal D_{m+1}}}
\exp\left(-\frac{\mathcal D_{m+1}}{2\varepsilon_{m+1}^2}\right),
\end{multline*}
where $\mathcal K=\mathcal K(\beta_1,\ldots,\beta_m)=\sum_{\nu=1}^m(2\nu+1)\beta_\nu$.
\end{proposition}

\begin{proof}
The boundary value problem corresponding to $W_m$ was derived in \cite{Gao:Hann:Torc:2003}, see also \cite{Naza:Niki:2004}.
Namely in Theorem \ref{sep} one should set $n=m+1$,
$$
k_\nu=
\begin{cases}
m-\nu & \text{for } \beta_\nu=0,\\
m+1+\nu & \text{for } \beta_\nu=1,
\end{cases}
\quad
\nu=1,\ldots,m;
\qquad
k_{m+1}=m,
$$
$$
k'_\nu=2m+1-k_\nu,\quad \nu=1,\ldots,m+1.
$$
This implies $\varkappa_0=\mathcal K+\frac{m(m+1)}{2}$, $\varkappa_1=\frac{(m+1)(3m+2)}{2}-\mathcal K$.

We substitute these quantities into (\ref{separated}) and obtain
$$
\Prob(\|W_m^{[\beta_1,\ldots,\beta_m]}\|_{\psi}\leq\varepsilon)\sim
\left(
\frac{\psi(1)}{\psi(0)}
\right)
^{-\frac{m+1}8+\frac{\mathcal K}{4(m+1)}}
\cdot
\Prob(\|W_m^{[\beta_1,\ldots,\beta_m]}\|_{\psi}\leq\varepsilon),
\quad
\varepsilon\to0.
$$

The asymptotics of probability $\Prob(\|W_m^{[\beta_1,\ldots,\beta_m]}\|\leq\varepsilon)$ was obtained in \cite[Proposition 1.5]{Naza:2003}.
\end{proof}

In a similar way, using \cite[Propositions 1.6 and 1.8]{Naza:2003}, we obtain the following relations.

\begin{proposition}
Let $B(t)$ be the Brownian bridge. Then, under assumptions of Proposition \ref{wiener}, the following relation holds:
\begin{multline*}
\Prob(\|B_m^{[\beta_1,\ldots,\beta_m]}\|_{\psi}\leq\varepsilon)
\sim
\left(
\psi(0)
\right)
^{\frac{m+1}8-\frac{\mathcal K}{4(m+1)}}
\left(
\psi(1)
\right)
^{\frac{\mathcal K+1}{4(m+1)}-\frac{m+1}8}
\times\\\times
\frac{(2m+2)^{\frac{m+1}2}\sqrt{2\sin\frac\pi{2m+2}}}{|\mathcal V(z_{m+1}^{1-3\beta_1},z_{m+1}^{2-5\beta_2},\ldots,z_{m+1}^{m+\beta_m})|}
\frac1{\sqrt{\pi\mathcal D_{m+1}}}
\exp\left(-\frac{\mathcal D_{m+1}}{2\varepsilon_{m+1}^2}\right).
\end{multline*}
\end{proposition}

\begin{proposition}
Let $\mathbb B_m(t)=\left(W_m(t)|W_j(1)=0,0\leq j\leq m\right)$ be conditional integrated Wiener process (see \cite{Lach:2002}). 
Then, under assumptions of Proposition \ref{wiener}, the following relation holds:
\begin{multline*}
\Prob(\|\mathbb B_m\|_\psi\leq\varepsilon)\sim
(\psi(0)\psi(1))^{\frac18}
\times\\\times
\frac{(2m+2)^{\frac{m}2+1}\cdot\left(\prod_{j=0}^m\frac{j!}{(m+1+j)!}\right)^{\frac12}}{|\mathcal V(1,z_{m+1},\ldots,z_{m+1}^m)|\sqrt{\pi\mathcal D_{m+1}}\cdot\varepsilon_{m+1}^{m(m+2)}}
\exp\left(-\frac{\mathcal D_{m+1}}{2\varepsilon_{m+1}^2}\right).
\end{multline*}
\end{proposition}

Let us introduce the notation
$$
\widetilde\varepsilon_n=\left(\varepsilon\sqrt{n\sin\frac\pi{2n}}\right)^{\frac1{2n-1}},
\quad
\widehat\varepsilon_n=\left(\varepsilon\sqrt{\frac{2n}{c_n}\sin\frac\pi{2n}}\right)^{\frac1{2n-1}},
\quad
c_n=\frac{2\sqrt\pi\Gamma(n)}{\Gamma(n-\frac12)}.
$$

The following relations can be obtained using \cite[Theorem 2.2]{Naza:2003}, \cite[Theorem 2.2]{Naza:2009} and \cite[Theorem 3.1]{Puse:2010}.

\begin{proposition}
\label{OU}
Let $U(t)$ be the Ornstein--Uhlenbeck process, i.e. the centered Gaussian process with the covariance function $\mathsf EU(t)U(s)=e^{-|t-s|}$. 
Then, under assumptions of Proposition \ref{wiener}, the following relation holds:
\begin{multline*}
\Prob(\|U_m^{[\beta_1,\ldots,\beta_m]}\|_{\psi}\leq\varepsilon)
\sim
\left(
\psi(0)
\right)
^{\frac{m+1}8-\frac{\mathcal K+1}{4(m+1)}}
\left(
\psi(1)
\right)
^{\frac{\mathcal K}{4(m+1)}-\frac{m+1}8}
\times\\\times
\frac{(2m+2)^{\frac{m+1}2}2\sqrt{e}\sqrt{\sin\frac\pi{2m+2}}}{|\mathcal V(z_{m+1}^{1-3\beta_1},z_{m+1}^{2-5\beta_2},\ldots,z_{m+1}^{m+\beta_m})|}
\frac{\widetilde\varepsilon_{m+1}^2}{\sqrt{\pi\mathcal D_{m+1}}}
\exp\left(-\frac{\mathcal D_{m+1}}{2\widetilde\varepsilon_{m+1}^2}\right).
\end{multline*}
\end{proposition}

\begin{proposition}
Let $S(t)=W(t+1)-W(t)$ be the Slepian process (see \cite{Slep:1961}). Then, under assumptions of Proposition \ref{wiener}, the following relation holds:
$$
\Prob(\|S_m^{[\beta_1,\ldots,\beta_m]}\|_{\psi}\leq\varepsilon)
\sim
\sqrt{\frac2e}
\Prob(\|U_m^{[\beta_1,\ldots,\beta_m]}\|_{\psi}\leq\varepsilon).
$$
\end{proposition}

\begin{proposition}
\label{Matern}
Let $M^{(n)}(t)$ be the Matern process (see \cite{Mate:1986}), i.e. the centered Gaussian process with the covariance function 
$$
\mathsf EM^{(n)}(t)M^{(n)}(s)=\frac{(n-1)!}{(2n-2)!}\exp(-|t-s|)\sum_{k=0}^{n-1}\frac{(n+k-1)!}{k!(n-k-1)!}(2|t-s|)^{n-k-1}.
$$
Then, under assumptions of Proposition \ref{wiener}, the following relation holds:
$$
\Prob(\|M^{(n)}\|_{\psi}\leq\varepsilon)
\sim
\left(
\psi(0)\psi(1)
\right)
^{-\frac{n}8}
\frac{\sqrt{2^{n^2+n+1}n^{n+1}e^n}}{|\mathcal V(1,z_n,\ldots,z_n^{n-1})|}
\frac{\widehat\varepsilon_n{}^{n^2+1}}{\sqrt{\pi\mathcal D_n}}
\exp\left(-\frac{\mathcal D_n}{2\widehat\varepsilon_n{}^2}\right).
$$
\end{proposition}

\begin{rem}
It is well known that $\{M^{(1)}(t),0\leq t\leq1\} \stackrel{law}{=} \{U(t),0\leq t\leq1\}$. It is easy to see that the formula from Proposition
\ref{OU} with $m=0$ coincides with the formula from Proposition \ref{Matern} with $n=1$.
\end{rem}

\section{Non-separated boundary conditions}

If some boundary conditions are not separated in the main terms, they can be split into pairs of the following form (see \cite[\S 18]{Naim:1969}):
\begin{equation}
\label{non-sep}
\begin{aligned}
&a v^{(\ell)}(0)+b v^{(\ell)}(1)+\sum_{j=0}^{\ell-1}\left(\alpha_{\nu j}v^{(j)}(0)+\gamma_{\nu j}v^{(j)}(1)\right)=0,\\
&b v^{(2n-\ell-1)}(0)+a v^{(2n-\ell-1)}(1)+\sum_{j=0}^{2n-\ell-2}\left(\alpha'_{\nu j}v^{(j)}(0)+\gamma'_{\nu j}v^{(j)}(1)\right)=0.
\end{aligned}
\end{equation}
We consider the case with a unique such pair.

\begin{theorem}
\label{non-separated}
Let the assumptions of Corollary~\ref{main} be satisfied. Suppose also that one pair of boundary conditions has the form (\ref{non-sep}) while other ones
are separated in the main terms\footnote{Note that the normalization condition implies that the numbers $k_\nu$ and $k'_\nu$, $\nu=1,\ldots,n-1$, differ
from $\ell$ and $2n-\ell-1$.}:
$$
\left.
\begin{aligned}
v^{(k_\nu)}(0)+\sum_{j=0}^{k_\nu-1}\left(\alpha_{\nu j}v^{(j)}(0)+\gamma_{\nu j}v^{(j)}(1)\right)=0,\\
v^{(k'_\nu)}(1)+\sum_{j=0}^{k'_\nu-1}\left(\alpha_{\nu j}v^{(j)}(0)+\gamma_{\nu j}v^{(j)}(1)\right)=0,
\end{aligned}
\right\}
\quad
\nu=1,\ldots,n-1.
$$
Denote by $\varkappa_0$ and $\varkappa_1$ the sums of orders of separated boundary conditions at zero and one, respectively:
$\varkappa_0=k_1+\ldots+k_{n-1}$, $\varkappa_1=k'_1+\ldots+k'_{n-1}$.
Then
\begin{multline}
\label{th-non-sep}
\lim_{\varepsilon\to0}
\frac{\Prob(\|X\|_{\psi_1}\leq\varepsilon)}{\Prob(\|X\|_{\psi_2}\leq\varepsilon)}=
\left(
\frac{\psi_2(0)}{\psi_1(0)}
\right)
^{\frac{\varkappa_0}{4n}-\frac{(n-1)(2n-1)}{8n}}
\left(
\frac{\psi_2(1)}{\psi_1(1)}
\right)
^{\frac{\varkappa_1}{4n}-\frac{(n-1)(2n-1)}{8n}}
\times\\\times
\left|
\frac
{\mathcal M_1 a^2\left(\frac{\psi_2(1)}{\psi_2(0)}\right)^{\frac{2n-2\ell-1}{4n}}+
\mathcal M_2 b^2\left(\frac{\psi_2(0)}{\psi_2(1)}\right)^{\frac{2n-2\ell-1}{4n}}}
{\mathcal M_1 a^2\left(\frac{\psi_1(1)}{\psi_1(0)}\right)^{\frac{2n-2\ell-1}{4n}}+
\mathcal M_2 b^2\left(\frac{\psi_1(0)}{\psi_1(1)}\right)^{\frac{2n-2\ell-1}{4n}}}
\right|^{\frac12},
\end{multline}
where
$$
\mathcal M_1=\mathcal V(\omega_1^{k_1},\ldots,\omega_1^{k_{n-1}},\omega_1^{\ell})\cdot
\mathcal V(\omega_1^{2n-\ell-1},\omega_1^{k'_1},\ldots,\omega_1^{k'_{n-1}}),
$$
$$
\mathcal M_2=\mathcal V(\omega_1^{k_1},\ldots,\omega_1^{k_{n-1}},\omega_1^{2n-\ell-1})\cdot
\mathcal V(\omega_1^{\ell},\omega_1^{k'_1},\ldots,\omega_1^{k'_{n-1}}).
$$
\end{theorem}

\begin{proof}
We have
\begin{multline*}
\theta_{-1}(\psi)=
(\psi(0))^{\frac{\varkappa_0}{2n}-\frac{(n-1)(2n-1)}{4n}}
(\psi(1))^{\frac{\varkappa_1}{2n}-\frac{(n-1)(2n-1)}{4n}}\times
\\
\times
\det
\mbox{\tiny$
\begin{pmatrix}
1 & \omega_1^{k_1} & \ldots & \omega_{n-1}^{k_1} & 0 & 0 & \ldots & 0\\
\vdots & \vdots & \ddots & \vdots & \vdots & \vdots & \ddots & \vdots\\
1 & \omega_1^{k_{n-1}} & \ldots & \omega_{n-1}^{k_{n-1}} & 0 & 0 & \ldots & 0\\
\widetilde\alpha_n & \widetilde\alpha_n\omega_1^{\ell} & \ldots & \widetilde\alpha_n\omega_{n-1}^{\ell} & \widetilde\gamma_n\omega_n^{\ell} & \widetilde\gamma_n\omega_{n+1}^{\ell} & \ldots & \widetilde\gamma_n\omega_{2n-1}^{\ell}\\
\widetilde\alpha_{n+1} & \widetilde\alpha_{n+1}\omega_1^{2n-2\ell-1} & \ldots & \widetilde\alpha_{n+1}\omega_{n-1}^{2n-\ell-1} & \widetilde\gamma_{n+1}\omega_n^{2n-\ell-1} & \widetilde\gamma_{n+1}\omega_{n+1}^{2n-\ell-1} & \ldots & \widetilde\gamma_{n+1}\omega_{2n-1}^{2n-\ell-1}\\
0 & 0 & \ldots & 0 & \omega_n^{k'_1} & \omega_{n+1}^{k'_1} & \ldots & \omega_{2n-1}^{k'_1}\\
\vdots & \vdots & \ddots & \vdots & \vdots & \vdots & \ddots & \vdots\\
0 & 0 & \ldots & 0 & \omega_n^{k'_{n-1}} & \omega_{n+1}^{k'_{n-1}} & \ldots & \omega_{2n-1}^{k'_{n-1}}\\
\end{pmatrix}
$}
=\\=
(\psi(0))^{\frac{\varkappa_0}{2n}-\frac{(n-1)(2n-1)}{4n}}
(\psi(1))^{\frac{\varkappa_1}{2n}-\frac{(n-1)(2n-1)}{4n}}
\cdot
(-1)^{\varkappa_1+2n-\ell-1}
\times\\\times
\left[
\widetilde\alpha_n\widetilde\gamma_{n+1}
\mathcal V(\omega_1^{k_1},\ldots,\omega_1^{k_{n-1}},\omega_1^{\ell})\cdot
\mathcal V(\omega_1^{2n-\ell-1},\omega_1^{k'_1},\ldots,\omega_1^{k'_{n-1}})
\right.
+\\+
\left.
\widetilde\alpha_{n+1}\widetilde\gamma_{n}
\mathcal V(\omega_1^{k_1},\ldots,\omega_1^{k_{n-1}},\omega_1^{2n-\ell-1})\cdot
\mathcal V(\omega_1^{\ell},\omega_1^{k'_1},\ldots,\omega_1^{k'_{n-1}})
\right]
.
\end{multline*}

Since
$$
\widetilde\alpha_n=a(\psi(0))^{\frac{2\ell+1-2n}{4n}}, \quad \widetilde\alpha_{n+1}=b(\psi(0))^{\frac{2n-2\ell-1}{4n}},
$$
$$
\widetilde\gamma_n=b(\psi(1))^{\frac{2\ell+1-2n}{4n}}, \quad \widetilde\gamma_{n+1}=a(\psi(1))^{\frac{2n-2\ell-1}{4n}},
$$
we have
\begin{multline*}
|\theta_{-1}(\psi)|=
(\psi(0))^{\frac{\varkappa_0}{2n}-\frac{(n-1)(2n-1)}{4n}}
(\psi(1))^{\frac{\varkappa_1}{2n}-\frac{(n-1)(2n-1)}{4n}}
\times\\\times
\left|
\mathcal M_1 a^2\left(\frac{\psi(1)}{\psi(0)}\right)^{\frac{2n-2\ell-1}{4n}}
+
\mathcal M_2 b^2\left(\frac{\psi(0)}{\psi(1)}\right)^{\frac{2n-2\ell-1}{4n}}
\right|
.
\end{multline*}
Now Corollary \ref{main} implies (\ref{th-non-sep}).
\end{proof}

The following relations can be obtained from Theorem~\ref{non-separated} using \cite[Theorem~3]{Puse:2010a} and \cite[Theorem~3.4]{Naza:2009}.

\begin{proposition}
Let $Y(t)$ be the Bogolyubov process (see \cite{Sank:1999,Sank:2005}). Then, under assumptions of Proposition \ref{wiener}, the following relation holds:
\begin{multline*}
\Prob
\{\|Y_m^{[\beta_1,\ldots,\beta_m]}\|_\psi\leq\varepsilon\}\sim
\left(\frac{\psi(0)}{\psi(1)}\right)^{\frac{m(m+2)}{8(m+1)}-\frac{\mathcal K}{4(m+1)}}
\times\\\times
\left|
\prod_{\nu=1}^m\big|1+z_{m+1}^{k_\nu}\big|^2 \left(\frac{\psi(0)}{\psi(1)}\right)^{\frac{1}{4(m+1)}}
+\prod_{\nu=1}^m\big|1+z_{m+1}^{2m+1-k_\nu}\big|^2 \left(\frac{\psi(1)}{\psi(0)}\right)^{\frac{1}{4(m+1)}}
\right|^{-\frac12}
\times\\\times
\frac{2^{m+2}(m+1)^{m+1}\sinh(\omega/2)}{|\mathcal V(z_{m+1}^{k_1},\ldots,z_{m+1}^{k_{m}})|}
\frac{\varepsilon_{m+1}}{\sqrt{\pi\mathcal D_{m+1}}}
\exp\left(-\frac{\mathcal D_{m+1}}{2\varepsilon_{m+1}^2}\right),
\end{multline*}
where $k_\nu=\nu-(2\nu+1)\beta_\nu$, $\nu=1,\dots,m$.
\end{proposition}

Consider multiply centered-integrated Brownian bridge:
$$
B_{\{0\}}(t)=B(t),\quad B_{\{l\}}(t)=\int_0^t\overline{B_{\{l-1\}}}(s)ds,\quad l\in\mathbb N.
$$

\begin{proposition}
Suppose that the function $\psi\in W_\infty^{m+2}(0,1)$ is bounded away from zero and satisfies the relation $\int_0^1\psi^{\frac1{2(m+2)}}(x)dx=1$.
Then the following relation holds:
\begin{multline*}
\Prob\{\|(B_{\{1\}})_m^{[\beta_1,\ldots,\beta_m]}\|_\psi\leq\varepsilon\}\sim
(\psi(0))^{\frac{m^2-3}{8(m+2)}-\frac{\widetilde{\mathcal K}}{4(m+2)}}
(\psi(1))^{\frac{\widetilde{\mathcal K}}{4(m+2)}-\frac{m^2+8m+3}{8(m+2)}}
\times
\\
\times
\left|
\prod_{\nu=1}^m\big|1+z_{m+2}^{k_\nu}\big|^2 \left(\frac{\psi(0)}{\psi(1)}\right)^{\frac{1}{4(m+2)}}
+
\prod_{\nu=1}^m\big|1+z_{m+2}^{2m+3-k_\nu}\big|^2 \left(\frac{\psi(1)}{\psi(0)}\right)^{\frac{1}{4(m+2)}}
\right|^{-\frac12}
\times\\\times
\frac{(2m+4)^{\frac{m+2}{2}}\sqrt{2\sin\frac{3\pi}{2m+4}}}{|\mathcal V(z_{m+2}^{k_1},\ldots,z_{m+2}^{k_{m}})|}
\frac{\varepsilon_{m+2}^{-2}}{\sqrt{\pi\mathcal D_{m+2}}}
\exp\left(-\frac{\mathcal D_{m+2}}{2\varepsilon_{m+2}^2}\right),
\end{multline*}
where $\widetilde{\mathcal K}=\widetilde{\mathcal K}(\beta_1,\ldots,\beta_m)=\sum_{\nu=1}^m(2\nu+3)\beta_\nu$ and
$k_\nu=\nu-(2\nu+3)\beta_\nu$, $\nu=1,\dots,m$.
\end{proposition}

Now we consider the case of boundary conditions periodic in the main terms. The following theorem can be easily derived from Corollary~\ref{main}.

\begin{theorem}
\label{periodic}
Let the assumptions of Corollary~\ref{main} be satisfied. Suppose also that the boundary conditions~(\ref{boundaryConditions}) have the form
$$
v^{(\nu)}(0)-v^{(\nu)}(1)+\sum_{j=0}^{\nu-1}\left(\alpha_{\nu j}v^{(j)}(0)+\gamma_{\nu j}v^{(j)}(1)\right)=0,
\quad
\nu=0,\ldots,2n-1.
$$
Then
\begin{multline*}
\lim_{\varepsilon\to0}
\frac{\Prob(\|X\|_{\psi_1}\leq\varepsilon)}{\Prob(\|X\|_{\psi_2}\leq\varepsilon)}=
\left(\frac{\psi_1(0)\psi_1(1)}{\psi_2(0)\psi_2(1)}\right)^{\frac{2n-1}8}
\times\\\times
\left|
\frac{\mathcal V((\psi_2(0))^{\frac{1}{2n}},(\psi_2(0))^{\frac{1}{2n}}\omega_1,\ldots,(\psi_2(0))^{\frac{1}{2n}}\omega_{n-1},
(\psi_2(1))^{\frac{1}{2n}}\omega_n,
\ldots,(\psi_2(1))^{\frac{1}{2n}}\omega_{2n-1})}
{\mathcal V((\psi_1(0))^{\frac{1}{2n}},(\psi_1(0))^{\frac{1}{2n}}\omega_1,\ldots,(\psi_1(0))^{\frac{1}{2n}}\omega_{n-1},
(\psi_1(1))^{\frac{1}{2n}}\omega_n,
\ldots,(\psi_1(1))^{\frac{1}{2n}}\omega_{2n-1})}
\right|^{\frac12}.
\end{multline*}
\end{theorem}

The following relation can be obtained from Theorem~\ref{periodic} using \cite[Theorem~3.2]{Naza:2009}.

\begin{proposition}
Under assumptions of Proposition \ref{wiener}, the following relation holds:
\begin{multline*}
\Prob(\|\overline{B_{\{m\}}}\|_{\psi}\leq\varepsilon)
\sim\left(\psi(0)\psi(1)\right)^{\frac{2m+1}8}
\times\\\times
\left|
\mbox{\tiny$
{\mathcal V((\psi(0))^{\frac{1}{2(m+1)}},(\psi(0))^{\frac{1}{2(m+1)}}z_{m+1},\ldots,(\psi(0))^{\frac{1}{2(m+1)}}z_{m+1}^{m},
(\psi(1))^{\frac{1}{2(m+1)}}z_{m+1}^{m+1},
\ldots,(\psi(1))^{\frac{1}{2(m+1)}}z_{m+1}^{2m+1})}
$}
\right|^{-\frac12}
\times\\\times
(2m+2)^{\frac{m+2}2}
\frac{\varepsilon_{m+1}^{-(2m+1)}}{\sqrt{\pi\mathcal D_{m+1}}}
\exp\left(-\frac{\mathcal D_{m+1}}{2\varepsilon_{m+1}^2}\right).
\end{multline*}
\end{proposition}


\begin{thebibliography}{99}
\bibitem{Naza:2003}
Nazarov~A.~I.,
\emph{On the sharp constant in the small ball asymptotics of some Gaussian processes under $L_2$-norm},
Probl. Mat. Anal.
{\bf 26} (2003)
179--214
(in Russian).
English transl.:
J. Math. Sci. (N.Y.)
{\bf 117} (2003)
4185--4210.

\bibitem{Naza:Puse:2009}
Nazarov~A.~I., Pusev,~R.~S.,
\emph{Exact small deviation asymptotics in $L_2$-norm for some weighted Gaussian processes},
Zap. Nauchn. Sem. POMI
{\bf 364} (2009)
166--199
(in Russian).
English transl.:
J. Math. Sci. (N.Y.)
{\bf 163} (2009)
409–-429.

\bibitem{Naim:1969}
Naimark,~M.~A.
\emph{Linear Differential Operators}, 2nd edn.,
Nauka, Moscow, 1969
(in Russian).
English transl. of the 1st edn.:
Part I. F. Ungar Publ. Co., NY, 1967.
Part II. F. Ungar Publ. Co., NY, 1968.

\bibitem{Niki:Puse:2012}
Nikitin~Ya.~Yu., Pusev,~R.~S.,
\emph{Exact small deviation asymptotics for some Brownian functionals},
Teor. Veroyatn. Primen.
{\bf 57} (2012)
98--123
(in Russian).
To be transl. in:
Theory Probab. Appl.
{\bf 57}.

\bibitem{Puse:2010}
Pusev~R.~S.,
\emph{Small deviations asymptotics for Mat\'ern processes and fields under weighted quadratic norm},
Teor. Veroyatn. Primen.
{\bf 55} (2010)
187--195
(in Russian).
English transl.:
Theory Probab. Appl.
{\bf 55} (2011)
164--172.

\bibitem{Puse:2010a}
Pusev~R.~S.,
\emph{Asymptotics of small deviations of the Bogoliubov processes with respect to a quadratic norm},
Teoret. Mat. Fiz.
{\bf 165} (2010)
134--144
(in Russian).
English transl.:
Theoret. Math. Phys.
{\bf 165} (2010)
1349--1358.

\bibitem{Sank:1999}
Sankovich~D.~P.,
\emph{Gaussian functional integrals and Gibbs equilibrium averages},
Teoret. Mat. Fiz.
{\bf 119} (1999)
345--352
(in Russian).
English transl.:
Theoret. Math. Phys.
{\bf 119} (1999)
670--675.

\bibitem{Sank:2005}
Sankovich~D.~P.,
\emph{The Bogolyubov functional integral},
Tr. Mat. Inst. Steklova
{\bf 251} (2005)
223--256
(in Russian).
English transl.:
Proc. Steklov Inst. Math.
{\bf 251} (2005)
213--245.

\bibitem{Fedo:1993}
Fedoryuk~M.~V.,
\emph{Asymptotical methods for linear ordinary differential equations},
2nd edn.,
LIBROKOM, Moscow, 2009
(in Russian).
English transl. of the 1st edn.:
{\it Asymptotic analysis: linear ordinary differential equations},
Springer, Berlin, 1993.

\bibitem{Shka:1982}
Shkalikov~A.~A.
\emph{Boundary-value problems for ordinary differential equations with a parameter in the boundary conditions},
Funktsional. Anal. i Prilozhen.
{\bf 16}:4 (1982)
92--93
(in Russian).
English transl.:
Funct. Anal. Appl.
{\bf 16}:4 (1982)
324--326.


\bibitem{Gao:Hann:Torc:2003}
Gao~F., Hannig,~J., Torcaso~T.,
\emph{Integrated Brownian motions and exact $L_2$-small balls},
Ann. Probab.
{\bf 31}
(2003),
1320--1337.

\bibitem{Gao:Hann:Torc:2003a}
Gao~F., Hannig~J., Torcaso,~T.,
\emph{Comparison theorems for small deviations of random series},
Electron. J. Probab.
{\bf 8}
(2003).

\bibitem{Lach:2002}
Lachal~A.,
\emph{Bridges of certain Wiener integrals. Prediction properties, relation with polynomial
interpolation and differential equations. Application to goodness-of-fit testing},
Bolyai Math. Studies X, Limit Theorems, Balatonlelle (Hungary), 1999,
1--51.
Budapest, 2002.

\bibitem{Li:1992}
Li~W.~V.,
\emph{Comparison results for the lower tail of Gaussian seminorms},
J. Theoret. Probab.
{\bf 5}
(1992),
1--31.

\bibitem{Li:Shao:2001}
Li~W.~V., Shao~Q.-M.,
\emph{Gaussian processes: inequalities, small ball probabilities and applications},
in Stochastic Processes: Theory and Methods,
Handbook of Statist.
{\bf 19},
533--597.
North-Holland, Amsterdam, 2001.

\bibitem{Lifs:1999}
Lifshits~M.~A.,
\emph{Asymptotic behavior of small ball probabilities},
in: Probability Theory and Mathematical Statistics: Proceedings of the Seventh International Vilnius Conference,
453--468,
TEV, Vilnius, 1999.

\bibitem{Lifs:2010}
Lifshits~M.~A.,
\emph{Bibliography on small deviation probabilities},
available at
\newline
{http://www.proba.jussieu.fr/pageperso/smalldev/biblio.pdf}

\bibitem{Mate:1986}
Mat\'ern~B.,
\emph{Spatial Variation},
Springer, Berlin, 1986.

\bibitem{Naza:2009}
Nazarov~A.~I.,
\emph{Exact $L_2$-small ball asymptotics of Gaussian processes and the spectrum of boundary-value problems},
J. Theoret. Probab.
{\bf 22}
(2009),
640--665.

\bibitem{Naza:Niki:2004}
Nazarov~A.~I., Nikitin~Ya.~Yu.,
\emph{Exact $L_2$-small ball behavior of integrated Gaussian processes and spectral asymptotics of boundary value problems},
Probab. Theory Related Fields
{\bf 129}
(2004),
469--494.

\bibitem{Slep:1961}
Slepian~D.,
\emph{First passage time for a particular Gaussian process},
Ann. Math. Stat.
{\bf 32}
(1961),
610--612.
\end{thebibliography}
\end{document}